\documentclass[12pt, reqno]{amsart}
\usepackage{}
\usepackage{mathrsfs}
\usepackage{amsfonts}
\usepackage{xcolor}
\usepackage{amsmath}
\usepackage{amsthm,enumerate}
\usepackage{amssymb}
\usepackage{cases}
\usepackage{bm}
\usepackage{cite}
\usepackage[pdfstartview=FitH, colorlinks=true,  linkcolor=red, citecolor=blue]{hyperref}
\numberwithin{equation}{section}
\newtheorem{theorem}{Theorem}[section]
\newtheorem{lemma}{Lemma}[section]

\newtheorem{remark}{Remark}[section]
\newtheorem{proposition}{Proposition}[section]
\newtheorem{definition}{Definition}[section]
\newtheorem{notation}{Notation}[section]
\newtheorem{condition}{Condition}[section]

\begin{document}
\title [Evolution of complete noncompact graphs]{Evolution of complete noncompact graphs by powers of curvature function }
\author{Guanghan Li }

\address{\parbox[l]{1\textwidth}{School of Mathematics and Statistics,
Wuhan University, Wuhan 430072, China}}
\email{ghli@whu.edu.cn }

\author{Yusha Lv}
\address{\parbox[l]{\textwidth}{School of Mathematics and Statistics,
Wuhan University, Wuhan 430072, China}}
\email{lvyushasx@163.com }



\maketitle

\begin{abstract}
This paper concerns the evolution of complete noncompact locally uniformly convex hypersurface in Euclidean space by curvature flow, for which the normal speed $\Phi$ is given by a power $\beta\geq 1$ of a monotone symmetric and  homogeneous of degree one function $F$ of the principal curvatures.
Under the assumption that $F$ is inverse concave and its dual function approaches zero on the boundary of positive cone, we prove that the complete smooth strictly convex solution exists and remains a graph until the maximal time of existence. In particular, if $F=K^{s/n}G^{1-s}$ for any $s\in(0, 1]$, where $G$ is a homogeneous of degree one, increasing in each argument and inverse concave curvature function, we prove that the complete noncompact smooth strictly convex solution exists and remains a graph  for all times.
\end{abstract}

\vskip 0.5cm
\section{Introduction}\label{sec1}

Let $\Sigma_0$ be a complete noncompact hypersurface embedded in $\mathbb R^{n+1}$ and  $X_0: M^n\rightarrow \mathbb R^{n+1}$ be a smooth immersion with $X_0(M)=\Sigma_0$. We consider a one-parameter family of smooth immersions $X: M\times[0, T)\rightarrow \mathbb R^{n+1}$ satisfying the following evolution equation
\begin{align}\label{1-1}
    \begin{cases}\frac{\partial}{\partial t}X(x,t)=-\Phi(F(\mathcal{W}(x,t)))\nu(x,t),  \\X(\cdot,0)=X_0(\cdot),\end{cases}
\end{align}
where $\nu(x,t)$ is the unit outward normal of the evolving hypersurface $\Sigma_t=X(M, t)$ at the point $X(x,t)$, $\mathcal W$ is the matrix of  Weingarten map of  $\Sigma_t$, $\Phi(F)=F^\beta~(\beta\geq 1)$ and  function $F(\mathcal W)$ satisfies the following conditions:
\begin{condition}\label{con1}
\begin{enumerate}[(i)]
\item $ F(\mathcal W)=f(\lambda(\mathcal W))$, where $\lambda(\mathcal W)$ gives the eigenvalues of $\mathcal W$ and $f$ is a smooth, symmetric function
defined on the positive cone $\Gamma_+=\{\lambda=(\lambda_1,\cdots,\lambda_n)\in \mathbb R^n: \lambda_i>0, i=1, \cdots, n\}$;

\item $f$ is strictly increasing in each argument: $\frac{\partial f}{\partial\lambda_i}>0$ on $\Gamma_+$, $\forall~i=1, \cdots, n$;

\item $f$ is homogeneous of degree one: $f(k\lambda)=k f(\lambda)$ for any $k>0$;

\item $f$ is strictly positive on $\Gamma_+$ and is normalized such that $f(1, \cdots, 1)=1$;

\item $f$ is inverse concave, that is, the function
\begin{align*}
  f_*(\lambda_1, \cdots, \lambda_n)=f(\lambda_1^{-1},\cdots,\lambda_n^{-1})^{-1}
\end{align*}
is concave;

\item $f_*$ approaches zero on the boundary of $\Gamma_+$.
\end{enumerate}
\end{condition}

 For compact convex hypersurface, problem \eqref{1-1} has been widely studied in the last decades. In \cite{H}, Huisken showed that any closed convex hypersurface evolving by the mean curvature flow contracts to a point in finite time, and become spherical in shape as the limit is approached. Later, this behavior were established for a wide range of flows where the speed is homogeneous of degree one in the principal curvatures, see \cite{A,A2,Chow,Chow1,H1,G1}. For higher homogeneity,  contracting flows and  constrained curvature flows were considered and studied in \cite{ACW,AM,AMZ,BW1,BS1,CS,GLW2,LL,LL1,S1,Tso}.

However, much less results are known when initial hypersurface is complete noncompact. In two fundamental papers \cite{EH} and \cite{EH1}, Ecker and Huisken studied the evolution of entire graph by the mean curvature. In \cite{EH}, they proved that if the initial hypersurface is a graph of locally Lipschitz continuous function and has linear growth rate for its height function, the solution exists for all times. They obtained some interior estimates in \cite{EH1} and applied them to prove that the hypothesis of linear growth in \cite{EH} is not necessary.  Later, Stavrou \cite{S2} proved the convergence to a selfsimilar profile of Lipschitz graphs having a unique cone at infinity, while Rasul \cite{R} obtained a convergence
result under a weaker oscillation condition than in \cite{EH}.

The result in \cite{EH1} can be extended to different ambient spaces. Unterberger \cite{U1,U2} proved the flow by the mean curvature of locally Lipschitz continuous entire radial graph over $\mathbb S^n_+$ in hyperbolic space $\mathbb H^{n+1}$ has a smooth
solution for all times, and each evolving hypersurface is an entire radial graph. Recently, in warped product space, Borisenko and Miquel \cite{BM} considered the flow by the mean curvature of a locally Lipschitz continuous graph on complete Riemannian manifold, and proved that the flow exists for all times and  evolving hypersurface remains a graph for all times.

The evolution of complete noncompact graphs by other special homogeneous function of degree equal to one has been considered, including $E^{1/k}_k$  \cite{H2} and  $\frac{E_k}{E_{k-1}}$ \cite{CD}, where $E_k$ is the elementary symmetric polynomial of degree $k$. In \cite{H2}, Holland derived gradient and curvature estimates for strictly $k$-convex solutions, and proved long time existence of the flow for $k$-convex initial data under assumption that initial graph function  $w_0(x)\rightarrow \infty $ as $|x|\rightarrow\infty$. Under the weak convexity assumption, Choi and Daskalopoulos  \cite{CD} proved the long time existence of complete convex solution for $\frac{E_k}{E_{k-1}}$-flow.
Recently, Alessandroni and Sinestrari \cite{AS} considered the evolution of entire convex graph by a general symmetric function $F$ of principal curvatures. If velocity $F$ is concave and inverse concave, they proved the solution exists for all times provided $F\geq \varepsilon H$ holds for some positive constant $\varepsilon$.

While for special homogeneous curvature function with higher degree, there are several results on curvature problems \eqref{1-1} for complete noncompact initial hypersurfaces. Under the assumption that initial graph is convex and satisfies a mild
condition on the oscillation of the normal, Schn\"{u}rer and Urbas \cite{SU} proved long time existence of convex graphs
evolving by powers of the Gauss curvature.  A similar result was obtained by Franzen \cite{F} for the flow by
powers of the mean curvature. Very recently, Choi, Daskalopoulos, Kim and Lee  \cite{CDKL} considered the evolution of complete noncompact locally uniformly convex hypersurface  by powers of Gauss curvature. Based on some a prior estimates for principal curvatures,
they proved that the solution of flow \eqref{1-1} exists and remains a graph for all times, without any assumption on the oscillation of the normal speed. We remark that the evolution of strictly mean convex entire graphs over $\mathbb R^n$ by inverse mean curvature flow was also considered by Daskalopoulos and Huisken in \cite{DH}, and they established the global existence of starshaped entire graphs with superlinear growth at infinity. More recently, Choi and Daskalopoulos \cite{CD1} studied the evolution of complete non-compact convex hypersurfaces in $\mathbb R^{n+1}$ by the inverse mean curvature flow. They established the long time existence of solutions and provided the characterization of the maximal time of existence in terms of the tangent cone at infinity of the initial hypersurface.

In this paper, we consider the evolution \eqref{1-1} of complete noncompact locally uniformly convex hypersurfaces by a power of curvature function satisfying Condition \ref{con1}. In order to formulate the main result of this work, it is necessary to recall some definitions as in \cite{CD,CDKL}.
\begin{definition}
We use $C^2_{\mathcal H}(\mathbb R^{n+1})$ to denote the class of second-order differentiable complete (either closed or non-compact) hypersurfaces
embedded in $\mathbb R^{n+1}$. Given any complete convex hypersurface $\Sigma$ and a point $p\in\Sigma$,  we define
the smallest principal curvature of $\Sigma$ at point $p$ by
$$\lambda_{\min}(\Sigma)(p)=\sup\left\{\lambda_{\min}(\Xi)(p): p\in\Xi\in C^2_{\mathcal H}(\mathbb R^{n+1}), \Sigma\subset {\rm the~convex~hall~of}~\Xi\right\},$$
and we say that
\begin{enumerate}[(i)]
\item $\Sigma$ is strictly convex, if $\lambda_{\min}(\Sigma)(p)>0$ holds for all $p\in \Sigma$;

\item $\Sigma$ is uniformly convex, if there is a constant $\varepsilon>0$ such that $\lambda_{\min}(\Sigma)(p)\geq\varepsilon$ for all $p\in\Sigma$;

\item $\Sigma$ is locally uniformly convex, if for any compact subset $\Omega\subset\mathbb R^{n+1}$, there is a constant $\varepsilon_{\Omega}>0$ such that $\lambda_{\min}(\Sigma)(p)\geq \varepsilon_{\Omega}$ for all $p\in\Sigma\cap\Omega$.
\end{enumerate}
\end{definition}

The first main result of this work is stated as follows:
\begin{theorem}\label{thm1}
Suppose curvature function $F$ satisfies Condition \ref{con1}. Let $\Sigma_0$ be a complete non-compact and locally uniformly convex hypersurface embedded in $\mathbb R^{n+1}$. Suppose $X_0: M^n\rightarrow \mathbb R^{n+1}$ is an immersion such that $\Sigma_0=X_0(M)$. Then for any $\beta\in[1, \infty)$, there exists a complete non-compact smooth and strictly convex solution $\Sigma_t=X(M^n, t)$ of \eqref{1-1}, which is the graph of some smooth and strictly convex function for all $t\in(0, T)$, where $T$ is the maximal time of existence of \eqref{1-1}.

In particular, if $\Sigma_0$ is an entire graph over $\mathbb R^n$, then the smooth strictly convex solution $\Sigma_t$ exists and remains a graph for all times $t\in(0, \infty)$.
\end{theorem}

In addition, for particular inverse concave curvature function $F=K^{s/n}G^{1-s}$~(See Remark \ref{rem1}), by constructing an appropriate barrier to guarantee each solution remains as a graph over the same domain, we have the following long time existence of solution to \eqref{1-1} for all times.
\begin{theorem}\label{thm4}
Suppose curvature function $G$ satisfies Condition \ref{con1} and $F=K^{s/n}G^{1-s}$ for any $s\in(0, 1]$. Let $\Sigma_0$ be a complete non-compact and locally uniformly convex hypersurface embedded in $\mathbb R^{n+1}$. Suppose $X_0: M^n\rightarrow \mathbb R^{n+1}$ is an immersion such that $\Sigma_0=X_0(M)$. Then for any $\beta\in[1, \infty)$, there exists a complete non-compact smooth and strictly convex solution $\Sigma_t=X(M^n, t)$ of \eqref{1-1}, which is the graph of some smooth and strictly convex function for all times $t\in(0, \infty)$.
\end{theorem}

\begin{remark}
The case $s=1$ of Theorem \ref{thm4} reduces to Theorem 1.1 in \cite{CDKL}. Compared with Theorem 1.1 in \cite{CDKL}, the power restriction $\beta\geq 1$
in Theorem \ref{thm4} comes from the estimation
of the local lower bound on the principal curvatures for general curvature function $F$ in Proposition \ref{prop-2}.
\end{remark}

As a byproduct of Theorem \ref{thm4}, we have the long time existence of a smooth solution $w: \Omega\times (0, \infty)\rightarrow \mathbb R$ to
the following fully nonlinear parabolic equation~(see formula \eqref{1-2} with $\Phi=K^{s\beta/n}G^{(1-s)\beta}$)
\begin{align*}
  \begin{cases}\frac{\partial w}{\partial t}=\frac{(\det D^2 w)^{s\beta/n}}{(1+|Dw|^2)^{\frac{(n+2)s\beta}{2n}-1}}G^{(1-s)\beta}(D^2 w, Dw, w, x, t),  \\ \mathop{\lim}\limits_{t\rightarrow 0}w(x, t)=w_0(x), \end{cases}
\end{align*}
where each $w(x, t)$ satisfies the conditions in Theorem \ref{thm2}~(see Section \ref{sec2}) and curvature function $G$ satisfies Condition \ref{con1}.

The rest of the paper is organized as follows. First we recall some notations, known
results and some basic evolution equations in Section \ref{sec2}. In Section \ref{sec3}, local a prior estimates for gradient function and the principal curvatures are established. We also prove the interior estimates for all derivatives of the second fundamental form by the inverse concavity of curvature function. Based on the interior estimates in previous section, Section \ref{sce4} is devoted to the proof of the existence of complete noncompact smooth solution, and the long time existence of solution for special inverse concave curvature function.

\vskip 0.5cm
\section{Notations and preliminary results}\label{sec2}

Let $X:M\rightarrow \mathbb{R}^{n+1}$ be a hypersurface of $\mathbb{R}^{n+1}$. The second fundamental form and the Weingarten map are denoted by $A=\{h_{ij}\}$ and $\mathcal W=\{g^{ik}h_{kj}\}=\{h^i_j\}$ respectively. The eigenvalues $\lambda_i, i\in\{1, \cdots, n\}$ of $\mathcal W$
are called the principal curvatures of $X(M)$ with respect to the induced metric $g=\{g_{ij}\}$. The trace of $\mathcal W$ with respect to $g$ is the mean
curvature $H$, and the Gauss curvature is
\begin{align*}
  K=\det(\mathcal W)=\det(h^i_j)=\frac{\det(h_{ij})}{\det(g_{ij})}=\prod^n_{i=1}\lambda_i.
\end{align*}
For a curvature function $F$ in Section \ref{sec1}, we shall use $\dot{F}^{kl}$ to indicate the matrix of the first order partial derivatives with respect
to the components of its argument
$$\frac{d}{ds}F(A+sB)\Big|_{s=0}=\dot{F}^{kl}\Big|_A B_{kl}.$$
Similarly the second order partial derivatives of $F$ are given by
$$\frac{d^2}{ds^2}F(A+sB)\Big|_{s=0}=\ddot{F}^{kl,rs}\Big|_A B_{kl}B_{rs}.$$
We also use the notations
$$\dot f^i(\lambda)=\frac{\partial f}{\partial\lambda_i}(\lambda)\quad\text{ and}\quad\ddot f^{ij}(\lambda)=\frac{\partial^2 f}{\partial\lambda_i\partial\lambda_j}(\lambda)$$
to denote the first and second derivatives of $f$ respect to $\lambda$.
In what follows, we will drop the arguments  when derivatives of $F$ and $f$ are evaluated at $\mathcal W$ and $\lambda(\mathcal W)$ respectively.
At any diagonal matrix $A$ with distinct eigenvalues, the second derivative $\ddot F$ in direction $B\in{\rm Sym}(n)$ can be expressed as follows (see \cite{A,A4}):
\begin{align}\label{2-8}
  \ddot F^{ij,kl}B_{ij}B_{kl}=\sum_{i,k}\ddot f^{ik}B_{ii}B_{kk}+2\sum_{i>k}\frac{\dot f^i-\dot f^k}{\lambda_i-\lambda_k}B_{ik}^2.
\end{align}
This formula makes sense as a limit in the case of any repeated values of $\lambda_i$.

The following  properties of inverse concave functions shall be needed.
\begin{lemma}[\cite{A4,AMZ}]\label{lem2}
If $f$ is inverse concave, then
 $\mathop{\sum}\limits_{i=1}^n\dot f^i\lambda_i^2\geq f^2$, and

\begin{align}
  \frac{\dot f^k-\dot f^l}{\lambda_k-\lambda_l}+\frac{\dot f^k}{\lambda_l}+\frac{\dot f^l}{\lambda_k}\geq 0,\quad\quad \forall~k\neq l. \label{2-6}
\end{align}
\end{lemma}

\begin{remark}\label{rem1}
There are many examples of inverse concave function with the dual function approaching zero on the boundary of positive cone, for example, $F=E_k^{1/k}~( k=1,\cdots,n)$, the power
means  $F=(\frac{1}{n}\mathop{\sum}\limits_{i}\lambda_i^r)^{\frac{1}{r}} ~(r>0)$, and convex function $F$. More examples can be constructed as follows:
If curvature functions $G_1$  and $G_2$ satisfy Condition \ref{con1}, then $F=G_1^{s}G_2^{1-s}$ satisfies Condition \ref{con1} for any $s\in[0, 1]$ (see \cite{A4,ALM} for more examples).
\end{remark}

In order to prove the main results, we need some extra notations as in \cite{CD,CDKL}.
\begin{notation}
\begin{enumerate}[(i)]
\item For set $\Sigma\subset\mathbb R^{n+1}$, we denote the convex hull of $\Sigma$ by
\begin{align*}
  {\rm Conv}(\Sigma)=\{\varepsilon x+(1-\varepsilon)y: x, y\in \Sigma, \varepsilon\in[0, 1]\}.
\end{align*}
\item Given a convex complete (either non-compact or closed) hypersurface $\Sigma$, if set $V$ is a
subset of ${\rm Conv}(\Sigma)$, we say $V$ is enclosed by $\Sigma$ and use the notation $V\preceq \Sigma$.
In particular, if $V\cap\Sigma=\varnothing$ and $V\preceq\Sigma$, we use $V\prec\Sigma$.\\
\item For a convex hypersurface $\Sigma$ and constant $\varepsilon>0$, we use $\Sigma^\varepsilon$ to denote its $\varepsilon$-envelope
\begin{align*}
  \Sigma^\varepsilon=\{Y\in\mathbb R^{n+1}: d(Y, \Sigma)=\varepsilon, Y\notin{\rm Conv}(\Sigma)\},
\end{align*}
where $d$ is the distance function.
\end{enumerate}
\end{notation}

For a locally uniformly convex hypersurface, we have the following theorem of Wu in \cite{W}.
\begin{theorem}[\cite{W}]\label{thm2}
Let $\Sigma$ be a complete and locally uniformly convex hypersurface embedded in $\mathbb R^{n+1}$, then there exists a function $w: \Omega\rightarrow \mathbb R$ defined on a convex open domain $\Omega\subset\mathbb R^n$ such that $\Sigma={\rm graph}~w$ and
\begin{enumerate}[(i)]
\item $w$ attains its minimum in $\Omega$ and $\mathop{\inf}\limits_{\Omega}w\geq0$;

\item if $\Omega\neq\mathbb R^n$, then $\mathop{\lim}\limits_{x\rightarrow x_0}w(x)=+\infty$ for all $x_0\in\partial\Omega$;

\item if $\Omega$ is unbounded, then $\mathop{\lim}\limits_{r\rightarrow +\infty}(\mathop{\inf}\limits_{\Omega\setminus B_r(0)}w)=\infty$.
\end{enumerate}
\end{theorem}

Let hypersurface $\Sigma$ be a graph given by function $w: \Omega\subset\mathbb R^n\rightarrow \mathbb R$, that is,
$$\Sigma=\{(x, w(x)): x\in\Omega\}.$$
Then the induced metric $g_{ij}$ and its inverse are given by
\begin{align*}
  g_{ij}=\delta_{ij}+w_iw_j\quad\quad{\rm and}\quad\quad g^{ij}=\delta^{ij}-\frac{w^iw^j}{1+|Dw|^2},
\end{align*}
where $w_i$ is the partial derivatives of $w$.
In addition, the unit outward normal is
\begin{align}\label{2-9}
  \nu=\frac{1}{\sqrt{1+|Dw|^2}}(Dw, -1).
\end{align}
The sign of the unit outward normal is chosen such that $\Sigma$ is convex if and only if the hessian of its graph representation $w(\cdot, t)$ is nonnegative.
After a standard computation, the second fundamental form can be expressed as
\begin{align*}
  h_{ij}=\frac{w_{ij}}{\sqrt{1+|Dw|^2}},
\end{align*}
which implies
\begin{align*}
  h^i_j=\frac{w_{jk}}{\sqrt{1+|Dw|^2}}\left( \delta^{ik}-\frac{w^iw^k}{1+|Dw|^2}\right).
\end{align*}
It follows from \eqref{1-1} and \eqref{2-9} that the parabolic system \eqref{1-1} is, up to tangential diffeomorphisms, equivalent to the following equation
\begin{align}\label{1-2}
  \begin{cases}\frac{\partial w}{\partial t}=\sqrt{1+|Dw|^2}\Phi(D^2 w, Dw, w, x, t),  \\ \mathop{\lim}\limits_{t\rightarrow 0}w(x, t)=w_0(x). \end{cases}
\end{align}

To ensure that evolving hypersurface stays a graph, we have to estimate $\langle \nu, \omega\rangle$ from below for some fixed vector $\omega\in\mathbb R^{n+1}, |\omega|=1$. Let us choose $\omega=-e_{n+1}$, and define the gradient function
$$v=\langle \nu, -e_{n+1}\rangle^{-1}=\sqrt{1+|Dw|^2},$$
and the height function
$$u(x, t)=\langle X(x, t), e_{n+1}\rangle.$$

We conclude this section by showing some evolution equations for important geometric quantities.
\begin{lemma}
Let $\Sigma_t$ be a complete strictly convex graph solution of \eqref{1-1}. Then the following evolution equations hold.
\begin{align}
\partial_tg_{ij}&=-2\Phi h_{ij}, \quad\quad\quad\quad\partial_t\nu=X_{*}(\nabla  \Phi), \nonumber\\
\partial_t \Phi&=\mathcal {L}\Phi+\Phi\dot\Phi^{ij}h_{ik}h^k_j, \label{2-1}\\
\partial_th_{ij}&=\mathcal Lh_{ij}+\ddot\Phi^{kl,mn}\nabla_ih_{kl}\nabla_jh_{mn}+\dot\Phi^{kl}h_{kp}h^p_lh_{ij}-(\beta+1)\Phi h_{ik}h^k_j,\label{2-4}\\
\partial_t b^{ij}&=\mathcal L b^{ij}-2b^{ip}b^{qk}b^{lj}\dot\Phi^{rs}\nabla_rh_{pq}\nabla_sh_{kl}-b^{ik}b^{jl}\ddot\Phi^{pq,rs}\nabla_kh_{pq}\nabla_lh_{rs}\nonumber\\
&\quad-b^{ij}\dot \Phi^{kl}h_{kp}h^p_l+(\beta+1)\Phi g^{ij},\label{2-5}\\
\partial_t u&=\mathcal Lu+(1-\beta)\Phi v^{-1},\label{2-2}\\
\partial_t v&=\mathcal L v-2v^{-1}|\nabla v|^2_{\mathcal L}-v\dot\Phi^{ij}h_{ik}h^k_j,\label{2-3}
\end{align}
where $b^{ij}=h_{ij}^{-1}, \mathcal L=\dot \Phi^{kl}\nabla_k\nabla_l$ and $|\nabla v|^2_{\mathcal L}=\dot \Phi^{kl}\nabla_k v\nabla_lv$.
\end{lemma}
\begin{proof}
The first four evolution equations under flow \eqref{1-1} follow from straightforward computations as in $\S 3$ of \cite{H}~(see also \cite{G,LL}). Now we prove the evolution equation for $b^{ij}$. The identity $b^{ik}h_{kj}=\delta^i_j$ implies
\begin{align*}
  \partial_t b^{pq}=-b^{pi}b^{qj}\partial_t h_{ij}\quad{\rm and}\quad \nabla b^{pq}=-b^{pi}b^{qj}\nabla h_{ij}.
\end{align*}
Therefore
\begin{align*}
  \nabla_r\nabla_s b^{pq}=-b^{pi}b^{qj}\nabla_r\nabla_s h_{ij}+2b^{ik}b^{pl}b^{jq}\nabla_sh_{ij}\nabla_rh_{kl},
\end{align*}
which implies
\begin{align*}
  \mathcal L b^{pq}=-b^{pi}b^{qj}\mathcal L h_{ij}+2b^{ik}b^{pl}b^{jq}\dot \Phi^{rs}\nabla_sh_{ij}\nabla_rh_{kl}.
\end{align*}
Combination of the above formulae with \eqref{2-4} gives
\begin{align*}
  \partial_tb^{pq}&=-b^{ip}b^{jq}\left( \mathcal Lh_{ij}+\ddot\Phi^{kl,mn}\nabla_ih_{kl}\nabla_jh_{mn}+\dot\Phi^{kl}h_{kr}h^r_lh_{ij}-(\beta+1)\Phi h_{ik}h^k_j\right)\\
  &=\mathcal L b^{pq}-2b^{ik}b^{pl}b^{jq}\dot \Phi^{rs}\nabla_rh_{ij}\nabla_sh_{kl}-b^{ip}b^{jq}\ddot\Phi^{kl,rs}\nabla_ih_{kl}\nabla_jh_{rs}\\
  &\quad-b^{pq}\dot\Phi^{kl}h_{kr}h^r_l+(\beta+1)\Phi g^{pq},
\end{align*}
which is equation \eqref{2-5}.

 Next, we give the proof of \eqref{2-2}. By direct computations we have
\begin{align*}
  \partial_tu=\langle\partial_t X, e_{n+1}\rangle=-\Phi\langle \nu, e_{n+1}\rangle,
\end{align*}
and
\begin{align*}
  \nabla_i\nabla_ju=\langle\bar\nabla_i\bar\nabla_jX, e_{n+1}\rangle=-h_{ij}\langle \nu, e_{n+1}\rangle.
\end{align*}
Then equation \eqref{2-2} follows from above two equations.

Last, we prove the evolution equation for gradient function $v$. From the evolution equation  for $\nu$, we have
\begin{align*}
  \partial_t v=-\partial_t\langle \nu, e_{n+1}\rangle^{-1}=v^2\langle\nabla\Phi, e_{n+1}\rangle,
\end{align*}
and
\begin{align*}
  \nabla_i\nabla_jv&=\nabla_i\left(v^2\langle \bar\nabla_j\nu, e_{n+1}\rangle\right)\\
  &=2v\nabla_iv\langle\bar\nabla_j \nu, e_{n+1}\rangle+v^2\langle\bar\nabla_i\bar\nabla_j\nu, e_{n+1}\rangle\\
  &=2v^{-1}\nabla_iv\nabla_jv+v^2\langle \nabla h_{ij}, e_{n+1}\rangle+vh_{ik}h^k_j,
\end{align*}
which implies formula \eqref{2-3}.
\end{proof}

\section{Local estimates}\label{sec3}

In this section, we will deduce interior a prior estimates for the gradient function, the principal curvatures and all the derivatives of the second fundamental form for solutions to flow \eqref{1-1}, under assumption that the initial hypersurface is smooth.

We begin by defining  cut-off functions
$$\varphi_\gamma=\left(R-u(x, t)-\gamma t\right)_+ \quad\quad{\rm and}\quad\quad \varphi=\left(R-u(x, t)\right)_+$$
 for some positive constants $R$ and $\gamma$. First, we have the following  local gradient estimate.
\begin{proposition}\label{prop-1}
Assume $\Sigma_0$ is a complete locally uniformly convex smooth hypersurface in $\mathbb R^{n+1}$, and let $\Sigma_t$ be a complete strictly convex smooth graph solution of
\eqref{1-1} defined on $M^n\times[0, T]$, for some $T>0$. Then, for some constants $\gamma>0$ and $R\geq \gamma$, we have
$$v(x, t)\varphi_\gamma(x, t)\leq R\max\left\{\mathop{\sup}\limits_{\bar \Sigma_0}v(x, 0),\quad\frac{\beta-1}{\gamma}\right\},$$
where $\bar \Sigma_0=\{x\in \Sigma_0: u(x, 0)\leq R\}$.
\end{proposition}
\begin{proof}
We derive from \eqref{2-2} and the definition of $\varphi_\gamma$ that
\begin{align*}
 \partial_t\varphi_\gamma=\mathcal L\varphi_\gamma+(\beta-1)\Phi v^{-1}-\gamma.
\end{align*}
Combining this with \eqref{2-3} we obtain
\begin{align*}
  \partial_t(\varphi_\gamma v)&=\varphi_\gamma(\mathcal L v-2v^{-1}|\nabla v|^2_{\mathcal L}-v\dot\Phi^{ij}h_{ik}h^k_j)+v(\mathcal L\varphi_\gamma+(\beta-1)\Phi v^{-1}-\gamma)\\
  &=\mathcal L(\varphi_\gamma v)-2\langle \nabla\varphi_\gamma, \nabla v\rangle_{\mathcal L}-2\varphi_\gamma v^{-1}|\nabla v|^2_{\mathcal L}-v\varphi_\gamma\dot\Phi^{ij}h_{ik}h^k_j+(\beta-1)\Phi -\gamma v\\
  &=\mathcal L(\varphi_\gamma v)-2v^{-1}\left\langle \nabla v,\nabla(\varphi_\gamma v)\right\rangle_{\mathcal  L}+(\beta-1)\Phi-\varphi_\gamma v\dot\Phi^{ij}h_{ik}h^k_j-\gamma v.
\end{align*}
It follows from Theorem \ref{thm2} that the cut-off function $\varphi_\gamma$ is compactly supported. Assume the function $\varphi_\gamma v$ attains its maximum at point $(x_0, t_0)$.
If $t_0=0$, then the result follows. Now let us assume $t_0>0$. Using $\beta\geq1$, we have the following inequality by the weak parabolic maximum principle
$$(\beta-1)\Phi\geq\gamma v+\varphi_\gamma v\dot\Phi^{ij}h_{ik}h^k_j.$$
Multiplying  above inequality by $R\Phi^{-1}$ and using $R\geq \gamma$ we have
\begin{align*}
  (\beta-1)R&\geq \gamma v(RF^{-\beta}+\beta\varphi_\gamma\dot F^{ij}h_{ik}h^k_jF^{-1})\\
  &\geq \gamma v(RF^{-\beta}+\beta\varphi_\gamma F)\\
  &\geq \gamma v\varphi_\gamma\left(\frac{F^{-\beta}}{\beta+1}+\frac{\beta F}{\beta+1}\right)\geq \gamma v\varphi_\gamma,
\end{align*}
where Lemma \ref{lem2} and Young's inequality are used in the second and last inequality, respectively. Then the assertion follows.
\end{proof}

Now we will show  local lower bounds on the principal curvatures in terms of the initial data. Here, a Pogorelov type
computation, which was introduced by Sheng,
Urbas and Wang in \cite{SUW} for the elliptic setting, appeared in \cite{CD,CDKL} is used. We begin by recall the  following known Euler's formula.

\begin{lemma}[\cite{CDKL}]\label{lem-1}
Let $\Sigma$ be a smooth strictly convex hypersurface. Assume smooth
immersion $X: M\rightarrow \mathbb R^{n+1}$ satisfies $\Sigma=X(M)$. Then, for all $x\in M$ and $i\in\{1, \cdots, n\}$, the following inequality holds
$$\frac{b^{ii}(x)}{g^{ii}(x)}\leq \frac{1}{\lambda_{\min}(x)},$$
where $\{b^{ij}\}$ is the inverse matrix of the second fundamental form $\{h_{ij}\}$.
\end{lemma}
\begin{proposition}\label{prop-2}
Assume $\Sigma_0$ is a complete locally uniformly convex smooth hypersurface in $\mathbb R^{n+1}$, and let $\Sigma_t$ be a complete strictly convex smooth graph solution of
\eqref{1-1} defined on $M^n\times[0, T]$, for some $T>0$.
Given a positive constant $R$, then for any $\sigma\in(0, 1)$, the following estimate holds
\begin{align*}
  \inf_{\{x\in\Sigma_t : u(x, t)\leq\sigma R, t\in[0, T]\}}\varphi\lambda_{\min}(x, t)\geq \inf_{\{x\in \Sigma_0: u(x, 0)\leq \sigma R\}}\varphi\lambda_{\min}(x, 0).
\end{align*}
\end{proposition}
\begin{proof}
Since cut-off function $\varphi$ is compactly supported by Theorem \ref{thm2}, then for  fixed $T>0$, the function $\varphi^{-1}\lambda^{-1}_{\min}$ attains its maximum on
$$\{\Sigma_t: u(x, t)\leq \sigma R, t\in[0, T]\}$$
 at point $(x_0, t_0)$. If $t_0=0$, the result follows. In what follows, we assume $t_0>0$.

Choose a chart $(U, \Psi)$ with $x_0\in\Psi(U)$ such that the covariant derivatives
$\{\nabla_i X(x_0, t_0)\}_{i=1,\cdots, n}$ form an orthonormal basis of $({T\Sigma_{t_0}})_{X(x_0, t_0)}$ satisfying
\begin{align*}
  g_{ij}(x_0, t_0)=\delta_{ij},\quad\quad h_{ij}(x_0, t_0)=\delta_{ij}\lambda_i(x_0, t_0),\quad\quad\lambda_1(x_0, t_0)=\lambda_{\min}(x_0, t_0).
\end{align*}
Then at point $(x_0, t_0)$, we have
$$ b^{11}(x_0, t_0)=\lambda^{-1}_{\min}(x_0, t_0),\quad\quad g_{11}(x_0, t_0)=1.$$
 Let us define the function
  \begin{align*}
  \rho=\varphi^{-1}\frac{b^{11}}{g^{11}}.
\end{align*}
For any  point $(x, t)\in\Psi(U)\times[0, T]$, the following inequality holds by Lemma \ref{lem-1}
\begin{align*}
  \rho(x, t)\leq \varphi^{-1}\lambda^{-1}_{\min}(x, t)\leq \varphi^{-1}\lambda^{-1}_{\min}(x_0, t_0)=\rho(x_0, t_0),
\end{align*}
which implies $\rho$ attains its maximum at point $(x_0, t_0)$.
From the evolution equation \eqref{2-2}, we have
\begin{align}\label{3-5}
 \partial_t \varphi=\mathcal L\varphi +(\beta-1)\Phi v^{-1}.
\end{align}
By the definition of $\rho$ and $\nabla g^{11}=0$, we derive that
\begin{align}\label{3-6}
  \mathcal L \rho&=\dot\Phi^{ij}\nabla_i\nabla_j\left(\varphi^{-1}\frac{b^{11}}{g^{11}}\right)\nonumber\\
  &=\dot\Phi^{ij}\nabla_i\left(-\varphi^{-2}\frac{\nabla_j\varphi b^{11}}{g^{11}}+\varphi^{-1}\frac{\nabla_jb^{11}}{g^{11}}\right)\nonumber\\
  &=\dot\Phi^{ij}\left(-\varphi^{-2}\frac{ b^{11}}{g^{11}}\nabla_i\nabla_j\varphi+2\varphi^{-3}\frac{ b^{11}}{g^{11}}\nabla_j\varphi\nabla_i\varphi-2\varphi^{-2}\frac{\nabla_j\varphi \nabla_i b^{11}}{g^{11}}+\frac{\varphi^{-1}}{g^{11}}\nabla_i\nabla_j b^{11}\right)\nonumber\\
  &=-\varphi^{-2}\frac{b^{11}}{g^{11}}\mathcal{L} \varphi +\frac{\varphi^{-1}}{g^{11}}\mathcal L b^{11}-2\varphi^{-1}\langle \nabla \varphi, \nabla \rho\rangle_{\mathcal L}.
\end{align}
Then at point $(x_0, t_0)$, from equations \eqref{2-5}, \eqref{3-5} and \eqref{3-6}, it follows that
\begin{align}
 \partial_t\rho&=-\varphi^{-2}\frac{b^{11}}{g^{11}}\partial_t \varphi +\varphi^{-1}\frac{\partial_t b^{11}}{g^{11}}-\varphi^{-1}\frac{b^{11}}{(g^{11})^2}\partial_tg^{11}\nonumber\\
 &=-\varphi^{-2}\frac{b^{11}}{g^{11}}\left(\mathcal L\varphi +(\beta-1)\Phi v^{-1}\right)+\frac{\varphi^{-1}}{g^{11}}\left( -b^{11}\dot \Phi^{kl}h_{kp}h^p_l+(\beta+1)\Phi g^{11}\right)\nonumber\\
 &\quad+\frac{\varphi^{-1}}{g^{11}}\left(\mathcal L b^{11}-2b^{1p}b^{qk}b^{l1}\dot\Phi^{rs}\nabla_rh_{pq}\nabla_sh_{kl}-b^{1k}b^{1l}\ddot\Phi^{pq,rs}\nabla_kh_{pq}\nabla_lh_{rs}\right)\nonumber\\
 &\quad-2\varphi^{-1}\frac{b^{11}}{(g^{11})^2}\Phi h^{11}\nonumber\\
 &=\mathcal L\rho+2\varphi^{-1}\langle \nabla \varphi, \nabla \rho\rangle_{\mathcal L}-\varphi^{-2}(\beta-1)\Phi v^{-1}b^{11}
+(\beta-1)\Phi\varphi^{-1} -b^{11}\varphi^{-1}\dot \Phi^{kl}h_{kp}h^p_l\nonumber\\
&\quad-\varphi^{-1}(b^{11})^2\left(2b^{ij}\dot\Phi^{kl}\nabla_1h_{ik}\nabla_1h_{jl}+\ddot\Phi^{pq,rs}\nabla_1h_{pq}\nabla_1h_{rs}\right).\label{3-7}
\end{align}
Now we estimate the terms on the last line  of \eqref{3-7} by a trick that appeared in the proof of Theorem 3.2 in \cite{BW1}. To make use the inverse concavity of $f$, let $\tau_i=\frac{1}{\lambda_i}$ and $f_*(\tau)=f(\lambda)^{-1}$. We can compute that
\begin{align*}
 \dot f^k=\frac{1}{f_*^2}\frac{\partial f_*}{\partial \tau_k}\frac{1}{\lambda_k^2}=\frac{\dot f^k_*}{f^2_*}\frac{1}{\lambda_k^2},
 \end{align*}
 and
 \begin{align}
  \ddot f^{kl}&=-\frac{\ddot f^{kl}_*}{f^2_*}\frac{1}{\lambda_l^2\lambda_k^2}+2\frac{1}{f^3_*}\frac{\dot f^k_*}{\lambda_k^2}\frac{\dot f^l_*}{\lambda_l^2}-2\frac{\dot f^{k}_*}{f^2_*}\frac{1}{\lambda_k^3}\delta_{kl}\nonumber\\
  &=-\frac{\ddot f^{kl}_*}{f^2_*}\frac{1}{\lambda_l^2\lambda_k^2}+\frac{2}{f}\dot f^k\dot f^l-2\frac{\dot f^k}{\lambda_k}\delta_{kl}. \label{3-4}
\end{align}
Therefore by equations \eqref{2-8} and \eqref{3-4}, we have
\begin{align*}
&2b^{ij}\dot\Phi^{kl}\nabla_kh_{i1}\nabla_lh_{j1}+\ddot\Phi^{pq,rs}\nabla_1h_{pq}\nabla_1h_{rs}\\
&\quad=2\beta F^{\beta-1}b^{ij}\dot F^{kl}\nabla_1h_{ik}\nabla_1h_{jl}+\beta F^{\beta-1}\ddot F^{pq,rs}\nabla_1h_{pq}\nabla_1h_{rs}+\beta(\beta-1)F^{\beta-2}(\nabla_1F)^2\\
&\quad=2\beta F^{\beta-1}b^{ij}\dot F^{kl}\nabla_1h_{ik}\nabla_1h_{jl}+\beta(\beta-1)F^{\beta-2}(\nabla_1F)^2\\
  &\quad\quad+\beta F^{\beta-1}\left(\ddot f^{kl}\nabla_1h_{kk}\nabla_1h_{ll}+\sum_{k\neq l}\frac{\dot f^k-\dot f^l}{\lambda_k-\lambda_l}(\nabla_1h_{kl})^2\right)\\
&\quad=2\beta F^{\beta-1}b^{ij}\dot F^{kl}\nabla_1h_{ik}\nabla_1h_{jl}
+\beta(\beta-1)F^{\beta-2}(\nabla_1F)^2 +\beta F^{\beta-1}\sum_{k\neq l}\frac{\dot f^k-\dot f^l}{\lambda_k-\lambda_l}(\nabla_1h_{kl})^2\\
&\quad\quad+\beta F^{\beta-1}\left(-\frac{ \ddot f_*^{kl}}{f_*^2}\frac{1}{\lambda_k^2\lambda_l^2}+\frac{2}{f}\dot f^k\dot f^l-2\frac{\dot f^k}{\lambda_k}\delta_{kl}\right)\nabla_1h_{kk}\nabla_1h_{ll}\\
&\quad\geq2\beta F^{\beta-1}b^{ij}\dot F^{kl}\nabla_1h_{ik}\nabla_1h_{jl}-\beta F^{\beta+1}\ddot f_*^{kl}\frac{1}{\lambda_k^2\lambda_l^2}\nabla_1h_{kk}\nabla_1h_{ll}\\
&\quad\quad-2\beta F^{\beta-1}\frac{\dot f^k}{\lambda_k}(\nabla_1h_{kk})^2-2\beta F^{\beta-1}\sum_{k\neq l}\frac{\dot f^k}{\lambda_l}(\nabla_1h_{kl})^2+\beta(\beta+1)F^{\beta-2}(\nabla_1F)^2\\
&\quad=2\beta F^{\beta-1}b^{ij}\dot F^{kl}\nabla_1h_{ik}\nabla_1h_{jl}-\beta F^{\beta+1}\ddot f_*^{kl}\frac{1}{\lambda_k^2\lambda_l^2}\nabla_1h_{kk}\nabla_1h_{ll}\\
&\quad\quad-2\beta F^{\beta-1}\sum_{k,l}\frac{\dot f^k}{\lambda_l}(\nabla_1h_{kl})^2+\beta(\beta+1)F^{\beta-2}(\nabla_1F)^2\\
&\quad=-\beta F^{\beta+1}\ddot f_*^{kl}\frac{1}{\lambda_k^2\lambda_l^2}\nabla_1h_{kk}\nabla_1h_{ll}+\beta(\beta+1)F^{\beta-2}(\nabla_1F)^2\geq0,
\end{align*}
where inequality \eqref{2-6} is used in the first inequality. By the fact
$$\dot F^{kl}h_{kp}h^p_l\geq h_{11}F,$$
equation \eqref{3-7} can be rewritten as follows
\begin{align*}
\partial_t\rho
 &\leq \mathcal L\rho+2\varphi^{-1}\langle \nabla \varphi, \nabla \rho\rangle_{\mathcal L}-\varphi^{-2}(\beta-1)\Phi v^{-1}b^{11}
+(\beta-1)\Phi\varphi^{-1} -b^{11}\varphi^{-1}\dot \Phi^{kl}h_{kp}h^p_l\\
&\leq \mathcal L\rho+2\varphi^{-1}\langle \nabla \varphi, \nabla \rho\rangle_{\mathcal L}+\varphi^{-2}(1-\beta)\Phi v^{-1}b^{11}
-\Phi\varphi^{-1}.
\end{align*}
Thus at maximum point $(x_0, t_0)$, we have
$$0\leq\partial_t\rho\leq \varphi^{-2}(1-\beta)\Phi v^{-1}b^{11}
-\Phi\varphi^{-1}<0,$$
which is a contradiction. Hence $t_0=0$ and the result holds.
\end{proof}

\begin{remark}
  In \cite{CDKL}, Choi, Daskalopoulos, Kim and Lee obtained local lower bounds on the principal curvature of strictly convex complete smooth graph solution of $K^\beta$-flow for all $\beta>0$ by considering the compactly supported function $\varphi_\gamma^{n(1+\frac{n}{\beta})}\frac{b^{11}}{g^{11}}$. However, for general inverse-concave curvature function $F$, the term $\ddot\Phi^{pq,rs}\nabla_1h_{pq}\nabla_1h_{rs}$ appeared in \eqref{3-7} can not be  estimated accurately as in \cite{CDKL}. Here we choose another compactly supported function $\varphi^{-1}\frac{b^{11}}{g^{11}}$, and the power $\beta\geq 1$ can not be weaken to $\beta>0$.
\end{remark}

Next we will derive  local upper bounds on the principal curvatures by using the assumption that dual function $F_*$ approaches zero on the boundary of positive cone. For this purpose, the local upper bound on velocity $F$ is needed.
\begin{proposition}\label{prop-3}
Assume $\Sigma_0$ is a complete locally uniformly convex smooth hypersurface in $\mathbb R^{n+1}$, and let $\Sigma_t$ be a complete strictly convex smooth graph solution of
\eqref{1-1} defined on $M^n\times[0, T]$, for some $T>0$. Then, given a constant $R>0$,  the following holds
\begin{align*}
  \frac{t}{1+t}F\varphi^2\leq C_0\theta^{1+\frac{1}{2\beta}},
\end{align*}
where  $C_0=2^{1+\frac{1}{2\beta}}\left(2\beta\Lambda\left(1+4\beta(\theta+1)\right)+R^2+2(\beta-1)R\right)$ and $\theta, \Lambda$ are given by
 \begin{align*}
   \theta=\mathop{\sup}\limits_{\{\Sigma_t: u(x, t)\leq R, t\in[0, T]\} }v^2,\quad\quad\quad\quad \Lambda=\mathop{\sup}\limits_{\{\Sigma_t: u(x, t)\leq R, t\in[0, T]\} }\lambda_{\min}^{-1}.
 \end{align*}
\end{proposition}
\begin{proof}
From equations \eqref{2-1} and \eqref{2-3}, we infer that
\begin{align*}
\partial_t\Phi^2&=\mathcal L \Phi^2-2|\nabla\Phi|^2_{\mathcal L}+2\Phi^2\dot\Phi^{ij}h_{ik}h^k_j,\\
  \partial_t v^2&=\mathcal L v^2-6|\nabla v|^2_{\mathcal L}-2v^2\dot\Phi^{ij}h_{ik}h^k_j.
\end{align*}
Using an idea of Caffarelli, Nirenberg and Spruck in \cite{CNS} (see also \cite{CD,CDKL,EH1}), we define function $\eta=\eta(v^2)$ by
$$\eta(v^2)=\frac{v^2}{2\theta-v^2}.$$
Then the evolution equation for $\eta$ is
\begin{align*}
  \partial_t\eta&=\eta'\partial_t v^2=\eta'(\mathcal L v^2-6|\nabla v|^2_{\mathcal L}-2v^2\dot\Phi^{ij}h_{ik}h^k_j)\\
  &=\mathcal L\eta-\eta''|\nabla v^2|^2_{\mathcal L}-6\eta'|\nabla v|^2_{\mathcal L}-2v^2\eta'\dot\Phi^{ij}h_{ik}h^k_j.
\end{align*}
Hence we have
\begin{align*}
  \partial_t(\Phi^2\eta)&=\Phi^2\left(\mathcal L\eta-\eta''|\nabla v^2|^2_{\mathcal L}-6\eta'|\nabla v|^2_{\mathcal L}-2v^2\eta'\dot\Phi^{ij}h_{ik}h^k_j\right)\\
  &\quad +\eta\left(\mathcal L \Phi^2-2|\nabla\Phi|^2_{\mathcal L}+2\Phi^2\dot\Phi^{ij}h_{ik}h^k_j\right)\\
  &=\mathcal L(\Phi^2\eta)-2\eta|\nabla \Phi|^2_{\mathcal L}-\Phi^2(4\eta''v^2+6\eta')|\nabla v|^2_{\mathcal L}\\
  &\quad+2\Phi^2(\eta-\eta'v^2)\dot\Phi^{ij}h_{ik}h^k_j-2\langle\nabla\eta, \nabla\Phi^2\rangle_{\mathcal L}.
\end{align*}
The last term can be estimated by
\begin{align*}
 -2\langle\nabla\eta, \nabla\Phi^2\rangle_{\mathcal L}&=-\langle\nabla\eta, \nabla\Phi^2\rangle_{\mathcal L}-\langle\nabla\eta, \eta^{-1}\nabla(\Phi^2\eta)\rangle_{\mathcal L}+\eta^{-1}\Phi^2|\nabla\eta|^2_{\mathcal L}\\
 &\leq-\eta^{-1}\langle\nabla\eta,\nabla(\Phi^2\eta)\rangle_{\mathcal L}+\frac{3}{2}\eta^{-1}\Phi^2|\nabla\eta|^2_{\mathcal L}+2\eta|\nabla\Phi|^2_{\mathcal L}.
\end{align*}
Thus the evolution equation for $\Phi^2\eta$ can be rewritten as
\begin{align*}
\partial_t(\Phi^2\eta)&\leq\mathcal L(\Phi^2\eta)-\eta^{-1}\langle\nabla\eta,\nabla(\Phi^2\eta)\rangle_{\mathcal L}+2\Phi^2(\eta-\eta'v^2)\dot\Phi^{ij}h_{ik}h^k_j\\
&\quad-\Phi^2\left(4\eta''v^2+6\eta'-6\eta^{-1}(\eta')^2v^2\right)|\nabla v|^2_{\mathcal L}.
\end{align*}
From the expression of $\eta$, we have
\begin{align*}
  \eta-\eta'v^2=-\eta^2, \quad\quad\quad\eta^{-1}\nabla\eta=4\theta\eta v^{-3}\nabla v,
\end{align*}
and
\begin{align*}
  4\eta''v^2+6\eta'-6\eta^{-1}(\eta')^2v^2=\frac{4\theta\eta}{(2\theta-v^2)^2}.
\end{align*}
Substituting these identities into the evolution equation for $\psi=\Phi^2\eta$ implies
\begin{align*}
  \partial_t \psi\leq \mathcal L\psi-4\theta\eta v^{-3}\langle\nabla v,\nabla \psi\rangle_{\mathcal L}-\frac{4\theta \psi}{(2\theta-v^2)^2}|\nabla v|^2_{\mathcal L}-2\eta\psi\dot\Phi^{ij}h_{ik}h^k_j.
\end{align*}
It follows from \eqref{3-5} that
\begin{align*}
  \partial_t\varphi^{4\beta}=\mathcal L\varphi^{4\beta}-4\beta(4\beta-1)\varphi^{4\beta-2}|\nabla\varphi|^2_{\mathcal L}+4\beta(\beta-1)\varphi^{4\beta-1}\Phi v^{-1}.
\end{align*}
Hence the following inequality holds
\begin{align*}
  \partial_t(\psi\varphi^{4\beta})&\leq\mathcal L(\psi\varphi^{4\beta})-4\theta\eta v^{-3}\varphi^{4\beta}\langle\nabla v,\nabla \psi\rangle_{\mathcal L}-\frac{4\theta\psi\varphi^{4\beta}}{(2\theta-v^2)^2}|\nabla v|^2_{\mathcal L}\\
  &\quad-2\eta \psi\varphi^{4\beta}\dot\Phi^{ij}h_{ik}h^k_j-4\beta(4\beta-1)\psi\varphi^{4\beta-2}|\nabla\varphi|^2_{\mathcal L}\\
  &\quad+4\beta(\beta-1)\varphi^{4\beta-1}\Phi v^{-1}\psi-2\langle\nabla \psi, \nabla \varphi^{4\beta}\rangle_{\mathcal L}.
\end{align*}
The last term can be rewritten as
\begin{align*}
  -2\langle\nabla \psi, \nabla \varphi^{4\beta}\rangle_{\mathcal L}=-2\varphi^{-4\beta}\langle\nabla\varphi^{4\beta}, \nabla(\psi\varphi^{4\beta})\rangle_{\mathcal L}+32\beta^2\varphi^{4\beta-2}\psi|\nabla\varphi|^2_{\mathcal L}.
\end{align*}
We also estimate
\begin{align*}
  &-4\theta\eta v^{-3}\varphi^{4\beta}\langle\nabla v,\nabla \psi\rangle_{\mathcal L}\\
  &\quad=-4\theta\eta v^{-3}\langle\nabla v,\nabla (\psi\varphi^{4\beta})\rangle_{\mathcal L}
+16\beta\theta\eta v^{-3}\psi\varphi^{4\beta-1}\langle\nabla v,\nabla \varphi\rangle_{\mathcal L}\\
  &\quad\leq-4\theta\eta v^{-3}\langle\nabla v,\nabla (\psi\varphi^{4\beta})\rangle_{\mathcal L} +\frac{4\theta \psi\varphi^{4\beta}}{(2\theta-v^2)^2}|\nabla v|^2_{\mathcal L}\\
  &\quad\quad+16\beta^2(2\theta-v^2)^2\theta \psi v^{-6}\eta^2\varphi^{4\beta-2}|\nabla\varphi|^2_{\mathcal L}\\
  &\quad=-4\theta\eta v^{-3}\langle\nabla v,\nabla (\psi\varphi^{4\beta})\rangle_{\mathcal L}+\frac{4\theta \psi\varphi^{4\beta}}{(2\theta-v^2)^2}|\nabla v|^2_{\mathcal L}\\
  &\quad\quad+16\beta^2\theta \psi v^{-2}\varphi^{4\beta-2}|\nabla\varphi|^2_{\mathcal L}.
\end{align*}
Combining above inequalities gives
\begin{align*}
  \partial_t(\psi\varphi^{4\beta})
  &\leq\mathcal L(\psi\varphi^{2\beta})-\langle4\theta\eta v^{-3}\nabla v+2\varphi^{-4\beta}\nabla\varphi^{4\beta},\nabla (\psi\varphi^{4\beta})\rangle_{\mathcal L}\\
  &\quad+\left(32\beta^2+16\beta^2\theta v^{-2}-4\beta(4\beta-1)\right)\psi\varphi^{4\beta-2}|\nabla\varphi|^2_{\mathcal L}\\
  &\quad+4\beta(\beta-1)\varphi^{4\beta-1}\Phi v^{-1}\psi-2\eta \psi\varphi^{4\beta}\dot\Phi^{ij}h_{ik}h^k_j.
\end{align*}
On the other hand, by
$$\nabla_i\varphi=-\nabla_i u=-\langle\bar\nabla_i X, e_{n+1}\rangle,$$
 we have
\begin{align*}
  |\nabla\varphi|^2_{\mathcal L}&=\beta F^{\beta-1}\dot F^{ij}\nabla_i\varphi\nabla_j\varphi\\
  &=\beta F^{\beta-1}\dot F^{ij}\langle\bar\nabla_iX, e_{n+1}\rangle\langle\bar\nabla_jX, e_{n+1}\rangle\\
  &\leq \sum_{p=1}^{n+1}\beta F^{\beta-1}\dot F^{ij}\langle\bar\nabla_iX, e_{p}\rangle\langle\bar\nabla_jX, e_p\rangle\\
  &=\beta F^{\beta-1}\dot F^{ij}g_{ij}\leq\frac{ \beta}{\lambda_{\min}}F^\beta\leq \beta \Lambda F^\beta,
\end{align*}
which implies that, by $v\geq 1$
\begin{align*}
  \left(32\beta^2+16\beta^2\theta v^{-2}-4\beta(4\beta-1)\right)\psi\varphi^{4\beta-2}|\nabla\varphi|^2_{\mathcal L}
  \leq 4\Lambda\beta^2(1+4\beta(\theta+1))\psi\varphi^{4\beta-2}F^{\beta}.
\end{align*}
Thus the evolution equation for $\psi\varphi^{4\beta}$ can be rewritten as
\begin{align*}
  \partial_t(\psi\varphi^{4\beta})&\leq\mathcal L(\psi\varphi^{4\beta})-\left\langle4\theta\eta v^{-3}\nabla v+2\varphi^{4\beta}\nabla\varphi^{4\beta},\nabla (\psi\varphi^{4\beta})\right\rangle_{\mathcal L}\\
  &\quad+4\Lambda\beta^2\left(1+4\beta(\theta+1)\right)\psi\varphi^{4\beta-2}F^{\beta}
  -2\eta \psi\varphi^{4\beta}\dot\Phi^{ij}h_{ik}h^k_j\\
  &\quad+4\beta(\beta-1)\varphi^{4\beta-1}\Phi v^{-1}\psi.
\end{align*}
Let $\xi=\frac{t}{1+t}$, then we have, by $\partial_t\xi=\frac{1}{(1+t)^2}\leq 1$,
\begin{align*}
  \partial_t(\psi\varphi^{4\beta}\xi^{2\beta})&\leq\mathcal L(\psi\varphi^{4\beta}\xi^{2\beta})-\langle4\theta\eta v^{-3}\nabla v+2\varphi^{4\beta}\nabla\varphi^{4\beta},\nabla (\psi\varphi^{4\beta}\xi^{2\beta})\rangle_{\mathcal L}\\
  &\quad-2\eta\psi\varphi^{4\beta}\xi^{2\beta}\dot\Phi^{ij}h_{ik}h^k_j+4\Lambda\beta^2\left(1+4\beta(\theta+1)\right)\psi\varphi^{4\beta-2}F^{\beta}\xi^{2\beta}\\
  &\quad+4\beta(\beta-1)\varphi^{4\beta-1}\Phi v^{-1}\psi\xi^{2\beta}+2\beta\xi^{2\beta-1}\psi\varphi^{4\beta}.
\end{align*}
 Since cut-off function $\varphi$ is compactly supported and $\xi=0$ when $t=0$, function $\psi\varphi^{4\beta}\xi^{2\beta}$ attains its maximum at some point $(x_0, t_0)$ for $t_0>0$. Thus the weak parabolic maximum principle implies, by the fact $\eta\geq\frac{1}{2\theta}$ and Lemma \ref{lem2} (i),
\begin{align}\label{3-1}
\frac{\beta}{\theta}\psi\varphi^{4\beta}\xi^{2\beta}F^{\beta+1}&\leq 2\eta\psi\varphi^{4\beta}\xi^{2\beta}\dot\Phi^{ij}h_{ik}h^k_j\nonumber\\
&\leq4\Lambda\beta^2\left(1+4\beta(\theta+1)\right)\psi\varphi^{4\beta-2}F^{\beta}\xi^{2\beta}
+2\beta\xi^{2\beta-1}\psi\varphi^{4\beta}\nonumber\\
&\quad+4\beta(\beta-1)\varphi^{4\beta-1}\Phi v^{-1}\psi\xi^{2\beta}.
\end{align}
Multiplying by $\frac{\theta}{\beta} \varphi^{-4\beta+2}\psi^{-1}\xi^{-2\beta+1}F^{-\beta}$ and noticing that $v\geq1, \xi\leq 1$ and $\varphi\leq R$, we have
\begin{align*}
  \varphi^2\xi F&\leq 4\beta\Lambda(1+4\beta(\theta+1))\theta\xi+2\theta\varphi^2 F^{-\beta}+4(\beta-1)\theta\varphi\xi v^{-1}\\
&\leq 4\beta\Lambda\left(1+4\beta(\theta+1)\right)\theta+2\theta\varphi^2 F^{-\beta}+4(\beta-1)\theta R\\
&= 4\beta\Lambda\left(1+4\beta(\theta+1)\right)\theta+2\theta(\varphi^2\xi F)^{-\beta}\xi^{\beta}\varphi^{2\beta+2}+4(\beta-1)\theta R\\
&\leq 4\beta\Lambda\left(1+4\beta(\theta+1)\right)\theta+2\theta R^2(\varphi^2\xi F)^{-\beta}R^{2\beta}+4(\beta-1)\theta R.
\end{align*}
If $\varphi^2\xi F\geq R^2$, above inequality shows
\begin{align*}
  \varphi^2\xi F\leq4\beta\Lambda\left(1+4\beta(\theta+1)\right)\theta+2\theta R^2+4(\beta-1)\theta R.
\end{align*}
Otherwise we can obtain, by $\theta\geq 1$,
$$\varphi^2\xi F\leq R^2\leq 4\beta\Lambda\left(1+4\beta(\theta+1)\right)\theta+2\theta R^2+4(\beta-1)\theta R.$$
Thus at point $(x_0, t_0)$, the following inequality holds
\begin{align*}
  \varphi^2\xi F\leq4\beta\Lambda\left(1+4\beta(\theta+1)\right)\theta+2\theta R^2+4(\beta-1)\theta R.
\end{align*}
Using that $\frac{1}{2\theta}\leq\eta\leq 1$ and $\psi\varphi^{4\beta}\xi^{2\beta}$ attains its maximum at point $(x_0, t_0)$, we conclude that, for all $(x, t)\in M\times[0, T)$, the following holds
\begin{align*}
  \frac{1}{2\theta}(F\varphi^2\xi)^{2\beta}&\leq \psi \varphi^{4\beta}\xi^{2\beta}=F^{2\beta}\eta\varphi^{4\beta}\xi^{2\beta}\leq (F\varphi^2\xi)^{2\beta}\\
&\leq \left (2\beta\Lambda\left(1+4\beta(\theta+1)\right)+R^2+2(\beta-1) R\right)^{2\beta}(2\theta)^{2\beta}.
\end{align*}
From this, the assertion follows.
\end{proof}

To establish the existence of smooth complete noncompact solution of flow \eqref{1-1}, the local estimates for all the derivatives of the second fundamental form are needed. Here we apply the Gauss map parametrization of convex hypersurface, which has been used widely in \cite{AMZ,BW1}, and write the flow \eqref{1-1} as a parabolic equation of the support function which is concave with respect to its arguments.
\begin{proposition}\label{prop-4}
Assume $\Sigma_0$ is a complete locally uniformly convex smooth hypersurface in $\mathbb R^{n+1}$, and let $\Sigma_t$ be a complete strictly convex smooth graph solution of
\eqref{1-1} defined on $M^n\times[0, T]$, for some $T>0$. Given a constant $R>0$, for any $k\geq 0$ and $\sigma\in(0, 1)$, we have
\begin{align*}
  \sup_{\{\Sigma_t: u(x, t)\leq \sigma R, t\in[0, T]\}}||\nabla^k A|| \leq C(n,~k,~R,~\sigma,~\beta,~\sup_{\bar\Sigma_0}v,~\sup_{\bar\Sigma_0}F,~\inf_{\bar\Sigma_0}\lambda_{\min}),
\end{align*}
where $\bar\Sigma_0=\{\Sigma_0:u(x, 0)\leq R\}$.
\end{proposition}
\begin{proof}
By Proposition \ref{prop-2}, we have principal curvatures are local bounded from below. And the local upper bounds for velocity $F$ can be obtained by Proposition \ref{prop-3}, which implies the dual function $F_*$ is bounded from below by a positive constant. Since $f_*$ approaches zero on the boundary of positive cone $\Gamma_+$, there exits a positive constant $c$ such that $\frac{1}{\lambda_i}\geq c$ for all $i$, that is, all the principal curvatures are local bounded from above.
Notice that $\bar\Sigma_t=\{\Sigma_t: u(x, t)\leq \sigma R\}$ is a closed convex hypersurface. Its support function is given by
$$S(z, t)=\sup\{\langle x, z\rangle: x\in\hat\Sigma_t, z\in \mathbb S^n\},$$
where $\hat\Sigma_t$ is a convex body enclosed by $\bar\Sigma_t$. Then hypersurface $\bar\Sigma_t$ can be given by the embedding ( ref. \cite{AMZ})
$$X(z, t)=S(z, t)z+ DS(z, t),$$
where $D$ is the gradient with respect to the standard  metric $\sigma_{ij}$ and connection on $\mathbb S^n$. The derivative of this map is given by
\begin{align*}
  \partial_iX=\tau_{ik}\sigma^{kl}\partial_lz,
\end{align*}
where $\tau_{ij}$ has the form
$$\tau_{ij}= D_i D_jS+\sigma_{ij}S.$$
In particular the eigenvalues of $\tau_{ij}$ with respect to the
metric $\sigma_{ij}$ are the inverses of the principal curvatures, or the principal radii of curvature.

Therefore the solution of (\ref{1-1}) is given, up to a time dependent tangential diffeomorphism, by solving the following scalar parabolic equation on $\mathbb S^n$
$$\partial_t S=-F^{-\beta}_*(\tau_{ij})\triangleq G( D^2S, DS, S, z, t),$$
for the support function $S$.
By the local bounds for principal curvatures, we already have local $C^2$ estimates on the support function $S$ and above formula is uniformly parabolic. Straightforward computations give
$$\dot G^{ij}=\frac{\partial G}{\partial(D^2_{ij}S)}=\beta F_*^{-\beta-1}\dot F_*^{pq}\frac{\partial\tau_{pq}}{\partial(D^2_{ij}S)}=\beta F_*^{-\beta-1}\dot F_*^{ij},$$
and
\begin{align}\label{E-1}
  \ddot G^{ij,kl}=-\beta(\beta+1)F_*^{-\beta-2}\dot F_*^{ij}\dot F_*^{kl}+\beta F_*^{-\beta-1}\ddot F_*^{ij,kl}.
\end{align}
By the concavity of $F_*$ and (\ref{E-1}), we have operator $G$ is concave with respect to $ D^2S$. From the local $C^2$ estimates on $S$ in space-time, we can apply the H\"older estimates of \cite{KS,TW} to obtain the $C^{2,\alpha}$ estimate on $S$ and $C^\alpha$ estimate on $\partial_tS$ in space-time. Therefore, by standard parabolic theory, we have all derivatives of the second fundamental form are bounded.
\end{proof}

\section{Existence of complete noncompact solution}\label{sce4}
Based on the local estimates in Section \ref{sec3}, in this section, we will establish the existence of the complete noncompact solution of \eqref{1-1} and the long time existence of solution for special inverse concave curvature function $F=K^{s/n}G^{1-s}$ for any $s\in(0, 1]$.

Given a locally uniformly convex hypersurface $\Sigma_0$, which is a graph of function $w$ defined on a convex open domain $\Omega\subset\mathbb R^n$, the following proposition shows the existence of complete noncompact solution $\Sigma_t$ of \eqref{1-1} on time interval $[0, T)$,  where $T$ depends on the given domain $\Omega$.
\begin{proposition}\label{prop-5}
Let $\Sigma_0$ be a complete noncompact and locally uniformly convex hypersurface embedded in $\mathbb R^{n+1}$. Suppose $X_0: M^n\rightarrow\mathbb R^{n+1}$ is an immersion such that $X_0(M)=\Sigma_0$. If $B_r(x_0)\subset\Omega_0$ for some $r>0$ and $x_0$, then there exists a solution $X(x, t): M\times(0, T)\rightarrow\mathbb R^{n+1}$ of \eqref{1-1} for some $T\geq (\beta+1)^{-1}r^{\beta+1}$ such that,
for each $t\in(0, T)$, the image $\Sigma_t=X(M, t)$ is a strictly convex smooth complete graph of
 function $w(\cdot, t): \Omega_t\rightarrow\mathbb R$ defined on a convex open $\Omega_t\subset\Omega_0$, and also $w(\cdot, t)$ and $\Omega_t$ satisfy
the conditions of $w_0$ and $\Omega_0$ determined by $\Sigma_0$ in Theorem \ref{thm2}.
\end{proposition}
\begin{proof}
The proof follows similar arguments as that of the proof of Theorem 5.1 in \cite{CDKL}. For
the convenience of readers, we sketch  the proof.

We first construct approximating sequence $\Gamma^i_t$. Let $w_0$ and $\Omega_0$ be determined by $\Sigma_0$ in Theorem \ref{thm2}, and assume $\mathop{\inf}\limits_{\Omega_0}w_0=0$.
For each $i\in\mathbb N$, let us define the approximate function
$$\tilde w^i_0(x)=w_0(x)+2/i$$
 with corresponding graph $\tilde \Sigma^i_0=\{(x, \tilde w^i_0(x)): x\in\Omega_0\}$.
Let  $\bar\Gamma^i_0$  denote the reflection of $\tilde\Sigma^i_0\cap(\mathbb R^n\times[0, i])$ over the $i$-level hyperplane, that is,
\begin{align*}
  \bar\Gamma^i_0=\{(x, h)\in\mathbb R^{n+1}: h\in\left(\tilde w^i_0(x), 2i-\tilde w^i_0(x)\right),\quad x\in\Omega_0, \tilde w^i_0(x)\leq i\}.
\end{align*}
It follows from the locally uniform convexity of $\Sigma_0$ that $\bar\Gamma^i_0$ is a uniformly convex closed hypersurface. Since $\bar\Gamma^i_0$ fails to be smooth at its intersection with the hyperplane $\mathbb R^n\times\{i\}$, we approximate $\bar\Gamma^i_0$ by its $\frac{1}{i}$-envelope $\Gamma^i_0$, which is a uniformly convex closed hypersurface
of class $C^{1,1}$. By Theorem 5 in \cite{AMZ} and approximation arguments, we can obtain that there exists a unique closed convex solution $\Gamma^i_t$ of \eqref{1-1} with initial data $\Gamma^i_0$ defined
for $t\in(0, T_i)$, where $T_i$ is the maximal time of existence. In addition,
the symmetry of  $\Gamma^i_t$ with respect to the hyperplane $\mathbb R^n\times\{i\}$ can be obtained by the uniqueness of solution. Thus its lower half $\Sigma^i_t=\Gamma^i_t\cap(\mathbb R^n\times[0, i])$ is a graph of some function $w^i(\cdot, t)$ defined on a convex set
$\Omega^i_t$, that is,
\begin{align*}
  \Sigma^i_t=\Gamma^i_t\cap(\mathbb R^n\times[0, i])=\{(x, w^i(x, t)): x\in\Omega_t^i\}.
\end{align*}
Let
\begin{align*}
  \Sigma_t=\partial\{\mathop{\cup}\limits_{i\in\mathbb N}{\rm Conv}(\Gamma^i_t)\},\quad\quad\Omega_t=\mathop{\cup}\limits_{i\in\mathbb N}\Omega^i_t,\quad t\in[0, T)
\end{align*}
where $T=\mathop{\sup}\limits_{i\in\mathbb N}T_i.$

By Proposition 6.3 in \cite{CD},  we can regularize $\Gamma^i_0$ by convolving its support  function with some compactly supported mollifiers on $\mathbb S^n$. Exactly as in \cite{CDKL}, we can prove that the interior estimates in Section \ref{sec3} hold in $\Gamma^i_t$ for cut-off function $\varphi_\gamma=(R-u(x, t)-\gamma t)_+$ with $R<i$. Therefore, $\Sigma^i_t$ is a strictly convex smooth graph of function $w^i(\cdot, t)$ for $t\in(0, T_i]$.

On the other hand, by the definition of $\Gamma^i_0$, we have $\Gamma^i_0\preceq\Gamma^{i+1}_0$. The comparison principle gives that
 $$\Gamma^i_t\preceq\Gamma^{i+1}_t\preceq\Sigma_0,$$
  which implies
\begin{align}\label{4-1}
  w_0(x)\leq w^{i+1}(x, t)\leq w^i(x, t).
\end{align}
 Thus $\mathop{\cup}\limits_{i\in\mathbb N}{\rm Conv}(\Gamma^i_t)$ is a convex body and $\Sigma_t=\partial(\mathop{\cup}\limits_{i\in\mathbb N}{\rm Conv}(\Gamma^i_t))$ is a complete convex hypersurface embedded in $\mathbb R^{n+1}$. From inequality \eqref{4-1}, we have
\begin{align*}
  w(x, t)=\lim_{i\rightarrow \infty}w^i(x, t), \quad\quad t\in(0, T).
\end{align*}
By the same manner as in the proof of Theorem 5.1 in \cite{CDKL}, we derive that $\Sigma_t$ is a complete noncompact strictly convex smooth graph solution. At last, the lower bound on the existence time $T$ can be achieved by the comparison principle.
\end{proof}

\if fase
\textbf{Step 1: The construction of the approximating sequence $\Gamma^i_t$.} Let $w_0$ and $\Omega_0$ be determined by $\Sigma_0$ as in Theorem \ref{thm2}, and assume $\mathop{\inf}\limits_{\Omega}w_0=0$.
For each $i\in\mathbb N$, let us define the approximate function $\tilde w^i_0(x)=w_0(x)+2/i$ with corresponding graph $\tilde \Sigma^i_0=\{(x, \tilde w^i_0(x)): x\in\Omega_0\}$.
Let  $\bar\Gamma^i_0$  denote the reflection of $\tilde\Sigma^i_0\cap(\mathbb R^n\times[0, i])$ over the $i$-level hyperplane, i.e.
\begin{align*}
  \bar\Gamma^i_0=\{(x, h)\in\mathbb R^{n+1}: h\in\left(\tilde w^i_0(x), 2i-\tilde w^i_0(x)\right),\quad x\in\Omega_0, \tilde w^i_0(x)\leq i\}.
\end{align*}
It follows from the locally uniform convexity of $\Sigma_0$ that $\bar\Gamma^i_0$ is uniformly convex closed hypersurface. Since $\bar\Gamma^i_0$ fails to be smooth at its intersection with hyperplane $\mathbb R^n\times\{i\}$, we approximate $\bar\Gamma^i_0$ by its $\frac{1}{i}$-envelope $\Gamma^i_0$, which is a uniformly convex closed hypersurface
of class $C^{1,1}$. By Theorem 5 in \cite{AMZ} and approximation arguments, we can obtain that there exists a unique closed convex solution $\Gamma^i_t$ of \eqref{1-1} with initial data $\Gamma^i_0$ defined
for $t\in(0, T_i)$, where $T_i$ is the maximal time of existence. In addition,
the symmetry of  $\Gamma^i_t$ with respect to the hyperplane $\mathbb R^n\times\{i\}$ can be obtained by the uniqueness of solution. Thus its lower half $\Sigma^i_t=\Gamma^i_t\cap(\mathbb R^n\times[0, i])$ is a graph for some function $w^i(\cdot, t)$ defined on a convex set
$\Omega^i_t$, that is,
\begin{align*}
  \Sigma^i_t=\Gamma^i_t\cap(\mathbb R^n\times[0, i])=\{(x, w^i(x, t)): x\in\Omega_t^i\}.
\end{align*}
Let
\begin{align*}
  \Sigma_t=\partial\{\mathop{\cup}\limits_{i\in\mathbb N}{\rm Conv}(\Gamma^i_t)\},\quad\quad\Omega_t=\mathop{\cup}\limits_{i\in\mathbb N}\Omega^i_t,\quad t\in[0, T)
\end{align*}
where $T=\mathop{\sup}\limits_{i\in\mathbb N}T_i.$

\textbf{Step 2: Local estimates for $\Gamma^i_t$.}
We will show the estimates in Section \ref{sec3} also hold in $\Gamma^i_t$ for cut-off function $\varphi_\gamma=(R-u(x, t)-\gamma t)_+$ with $R<i$.

Regarding  $(0, i)$ as a origin, then the support function $s^i_0(\nu)=\mathop{\max}\limits_{Y\in\Gamma^i_0}\langle\nu, Y\rangle$ of $\Gamma^i_0$ is positive and symmetric. Then for small $\varepsilon$, the convolution $s^{i,\varepsilon}_0=s^i_0*\eta_{\varepsilon}$ for some compactly supported mollifiers $\eta_{\varepsilon}$ on $\mathbb S^n$ is the support function of a strictly convex closed smooth hypersurface $\Gamma^{i,\varepsilon}_0$ by Proposition 6.3 in \cite{CD}. Let $\Gamma^{i,\varepsilon}_t$ be
the unique strictly convex closed smooth solution of \eqref{1-1} with initial data $\Gamma^{i, \varepsilon}_0$ defined for $t\in[0, T^{i,\varepsilon}_0)$~(see Theorem 1.1 in \cite{LL}), which guarantee $\Gamma^{i,\varepsilon}_t$ is also symmetric with
respect to $\mathbb R^n\times\{i\}$. Therefore, the lower half of $\Gamma^{i,\varepsilon}_t$ remains as a graph for some function $w^{i, \varepsilon}(\cdot, t)$.

What's more, by  Proposition 6.3 in \cite{CD}, $\Gamma^{i,\varepsilon}_0$ converges to $\Gamma^i_0$ and the
following inequality holds
\begin{align*}
  \mathop{\lim\inf}\limits_{\varepsilon\rightarrow 0}\lambda_{\min}(\Gamma^{i,\varepsilon}_0)(X_\varepsilon)\geq\lambda_{\min}^{\rm loc}(\Gamma_0^i)(X),
\end{align*}
where $\{X_{\varepsilon}\}$ is a set of points $X_\varepsilon\in\Gamma^{i,\varepsilon}_0$ converging to $X\in\Gamma^{i}_0$ as $\varepsilon\rightarrow 0$, and
\begin{align*}
  \lambda_{\min}^{\rm loc}(\Gamma_0^i)(X)=\mathop{\lim\inf}\limits_{s\rightarrow 0}\{\lambda_{\min}(\Gamma_0^i)(Y): Y\in\Gamma_0^i\cap B_s^{n+1}(X)\}.
\end{align*}
Thus for $R<i$, $\bar\Gamma^{i,\varepsilon}_0=\{\Gamma^{i,\varepsilon}_0: u^{i,\varepsilon}(x,0 )\leq R\}$ and $\bar\Gamma^{i}_0=\{\Gamma^{i}_0: u^{i}(x,0 )\leq R\}$, we have
\begin{align*}
  \mathop{\lim\inf}\limits_{\varepsilon\rightarrow 0}(\inf_{X_\varepsilon\in \bar\Gamma^{i,\varepsilon}_0}\lambda_{\min}(\Gamma^{i,\varepsilon}_0)(X_\varepsilon))\geq\inf_{X\in \bar\Gamma^{i}_0}\lambda_{\min}^{\rm loc}(\Gamma_0^i)(X),
\end{align*}
and
\begin{align*}
  \mathop{\lim\sup}\limits_{\varepsilon\rightarrow 0}(\sup_{X_\varepsilon\in \bar\Gamma^{i,\varepsilon}_0}v(\Gamma^{i,\varepsilon}_0)(X_\varepsilon))\leq\sup_{X\in \bar\Gamma^{i}_0}v(\Gamma_0^i)(X).
\end{align*}
Hence, the uniform gradient estimate and the uniform lower bound for principal curvatures of $\Gamma^{i, \varepsilon}_t$ can be obtained by Proposition \ref{prop-1} and Proposition \ref{prop-2} respectively. Due to dual function $F_*$ approaches zero on the boundary of positive cone,
 all the principal curvatures of $\Gamma^{i,\varepsilon}_t$ are uniformly bonded from above by Proposition  \ref{prop-3}.  Then Proposition \ref{prop-4} imply the
local uniformly estimates for all derivatives of the second fundamental form of $\Gamma^{i,\varepsilon}_t$. Therefore, the limit
$\Gamma^i_t$ is a strictly convex smooth closed graph for $t\in(0, T_i]$.

\textbf{Step 3: Passing $\Gamma^i_t$ to the limit $\Sigma_t$.} By the definition of $\Gamma^i_0$, we have $\Gamma^i_0\preceq\Gamma^{i+1}_0$. The comparison principle implies $\Gamma^i_t\preceq\Gamma^{i+1}_t\preceq\Sigma_0$, which implies $w_0(x)\leq w^{i+1}(x, t)\leq w^i(x, t)$. Thus $\mathop{\cup}\limits_{i\in\mathbb N}{\rm Conv}(\Gamma^i_t)$ is a convex body and $\Sigma_t=\partial(\mathop{\cup}\limits_{i\in\mathbb N}{\rm Conv}(\Gamma^i_t))$ is a complete convex hypersurface embedded in $\mathbb R^{n+1}$.

For any constant $R>0, t_0>0, \sigma\in(0, 1)$ and $i\gg R$, we can derive a uniform gradient bound of $\Gamma^i_t$ for cut-off function $\varphi_\gamma$ with $\gamma=\min\{R, t_0^{-1}\sigma\}$ for $i>R$ and $t\in[0, t_0]$ in $\bar\Gamma^i_t=\{\Gamma^i_t: u^i(x, t)\leq \sigma R\}$ by Proposition \ref{prop-1}. So the gradient function $v(\Sigma_t)$ of $\Sigma_t$ is
bounded in $\bar\Sigma_t=\{\Sigma_t: u(x, t)\leq \sigma R\}$. Thus, $\Sigma_t$ is a complete convex graph of function $w(\cdot, t)$ defined on a convex open set $\Omega_t$.
Since $w_0(x)\leq w(x, t)$ by $w_0(x)\leq w^{i+1}(x, t)\leq w^i(x, t)$, we have $w(x, t)$ and $\Omega_t$ satisfy the same
conditions  of $w_0$ and $\Omega_0$ determined by $\Sigma_0$ in Theorem \ref{thm2}.
Moreover, by Proposition \ref{prop-2}, Proposition \ref{prop-3} and the property of curvature function $F$, the principal curvatures of $\Gamma^i_t$ are uniform bounded from below and above in $\bar\Gamma^i_t=\{\Gamma^i_t: u^i(x, t)\leq \sigma R\}$ for $i\gg R, \sigma\in(0, 1)$ and $t\in [t_1, t_0]$ with $0<t_1<t_0<T$.

\textbf{Step 4: Passing $w^i$ to the limit $w$.} Since $\Sigma^i_t$ is the graph of function $w^i(\cdot, t)$ and $w_0(x)\leq w^{i+1}(x, t)\leq w^i(x, t)$, the following holds
\begin{align*}
  w(x, t)=\lim_{i\rightarrow \infty}w^i(x, t), \quad\quad t\in(0, T).
\end{align*}
Choosing $t_1<t_0\in(0, T)$ and $x_0\in\Omega_{t_0}$. Since $\Omega_{t_0}$ is open, there
exists  $r_0>0$ such that $\overline{B_{r_0}(x_0)}\subset \Omega_{t_0}$ and $r_0^2<t_0-t_1$. In addition $\Omega^i_{t_0}$ monotone converges to $\Omega_{t_0}$ by $\Gamma^i_t\preceq\Gamma^{i+1}_t\preceq\Sigma_0$. Hence, for sufficiently large $i_0$, we have $B_{r_0}(x_0)\subset\Omega^{i_0}_{t_0}$, which implies $u^{i_0}(x, t_0)=w^{i_0}(x, t_0)\leq i_0$. Thus, for all $i\geq i_0, t\in[t_1, t_0]$ and $x\in B_r(x_0)$,  we have
\begin{align*}
  0\leq w_0(x)\leq w^i(x, t)\leq w^i(x, t_0)\leq w^{i_0}(x, t_0)\leq i_0,
\end{align*}
that is,
 $$0\leq u_0(x)\leq u^i(x, t)\leq u^{i_0}(x, t_0)\leq i_0.$$
 Moreover the principal curvatures of $\Sigma_t$ are local uniformly bounded by Proposition \ref{prop-2} and Proposition \ref{prop-3}.
Besides above inequality and Proposition \ref{prop-4} imply the limit $\Sigma_t$ is a complete noncompact strictly convex smooth graph solution.

Last let us prove the maximal times of existence satisfies $T\geq(\beta+1)^{-1}r^{\beta+1}$.
Since $B_r(x_0)\subset\Omega_0$, there exists  constant $h_0>$ such that $B^{n+1}_r(x_0, h_0)\preceq\Sigma_0$. Set $X_0=(x_0, h_0)$. By $\Gamma^i_t\preceq\Gamma^{i+1}_t\preceq\Sigma_0$, we can choose $i\geq r+h_0$ such that $B^{n+1}_r(X_0)\preceq\Gamma^i_0$. On the other hand, $\partial B^{n+1}_{r_i(t)}(X_0)$ is a solution of \eqref{1-1}, where $r_i(t)=(r^{\beta+1}-(\beta+1)t)^{\frac{1}{\beta+1}}$. Then the comparison principle implies $\partial B^{n+1}_{r_i}(X_0)\preceq \Gamma^i_t$. Observe that $\Gamma^i_t$ exists while $r_i>0$, we have $T\geq \frac{1}{\beta+1}r^{\beta+1}$.
\fi

{\renewcommand{\proofname}{\textbf{Proof of Theorem \ref{thm1}}}
\begin{proof}
It follows from the locally uniformly convexity of $\Sigma_0$ that there exists function $w_0$ defined on a convex open domain $\Omega_0\subset\mathbb R^n$ such that $\Sigma_0={\rm graph}~w_0$ by Theorem \ref{thm2}. As a result, there exist point $x_0$ and $r>0$ such that $B_r(x_0)\subset\Omega_0$, and the first part of Theorem \ref{thm1} follows by Proposition \ref{prop-5}.

If initial hypersurface $\Sigma_0$ is an entire graph over $\mathbb R^n$, then for any $r>0$, we have $B_r(x_0)\subset \mathbb R^n$. And the long time existence of solution follows by the lower bound on the existence time $T\geq(\beta+1)^{-1}r^{\beta+1}$ in Proposition \ref{prop-5}.
\end{proof}}

To prove the long time existence of complete noncompact solution of \eqref{1-1} for special inverse concave curvature function $F=K^{s/n}G^{1-s}$ in Theorem \ref{thm4}, we will construct an appropriate barrier (see Theorem 5.4 in \cite{CDKL}) to guarantee that each $\Sigma_t$ remains as a graph over the same domain $\Omega$ for all $t\in(0, T)$, which implies $T=\infty$ independently from the domain $\Omega$.
\begin{theorem}\label{thm3}
  Suppose curvature function $G$ satisfies Condition \ref{con1} and $F=K^{s/n}G^{1-s}$ for any $s\in(0, 1]$. Let $\Sigma_0$ be a complete noncompact and locally uniformly convex hypersurface embedded in $\mathbb R^{n+1}$. Assume that $\Sigma_t=\{(x, w(x, t)): x\in\Omega_t, t\in(0, T)\}$ is a strictly convex smooth complete graph solution of \eqref{1-1} such that $\Omega_t, w(\cdot, t)$ and $\Sigma_t$ satisfy the conditions of $\Omega_0$ and $w_0$ determined by $\Sigma_0$ in Theorem \ref{thm2}. Then for any closed ball $\overline {B_{R_0}(x_0)}\subset \Omega_0$ and any $t_0\in(0, T)$, we have $B_{R_0}(x_0)\subset \Omega_{t_0}$.
\end{theorem}
\begin{proof}
Without loss of generality, we may assume that $x_0=0$ and $R_0<1$. From $\overline{B_{R_0}(0)}\subset\Omega_0$,
it follows that there exists a constant $l_0\geq 0$ such that $\overline{B_{R_0}(0)}\preceq L_{l_0}(\Sigma_0)=B_{l_0}(0)$. For given constants $l\geq l_0+1$, $\sigma\in(0, 1)$ and $\delta>0$ sufficiently small such that
$$\delta+2^{(1-s+s/n)\beta+2}(1-\sigma)^{-\beta}n^{(1-s)\beta}R_0^{-\beta}\delta^{\frac{s\beta}{n}}t_0\leq\sigma R_0,$$
let us define the function $\phi^{\delta, l}: [l-1, l]\times[0, t_0]\rightarrow \mathbb R$ by
\begin{align*}
  \phi^{\delta, l}(h, t)=R_0-\delta(h-l)^2-2^{(1-s+s/n)\beta+2}(1-\sigma)^{-\beta}n^{(1-s)\beta}R_0^{-\beta}\delta^{\frac{s\beta}{n}}t.
\end{align*}
We denote the inverse function of  $\phi(\cdot, t)$ by $\phi^{-1}(\cdot, t)$, and denote the graph of the rotationally
symmetric function $\phi^{-1}(|x|, t)$ by
\begin{align*}
  \Psi^{\delta, l}_t=\{(x, h): |x|=\phi^{\delta, l}(h, t)\}.
\end{align*}
It is easy to see that $\Psi^{\delta, l}_0\preceq \Sigma_0$. We claim that $\Psi^{\delta, l}_t$ defines a supersolution of \eqref{1-1}.

In fact, let $\phi^{-1}_{r}$ and $\phi^{-1}_{rr}$ denote the first and second derivatives of $\phi^{-1}(r, t)$ with respect to $r$. Then the Gauss curvature $K$ and the mean curvature $H$ of $\Psi^{\delta, l}_t$ satisfy the following inequalities respectively
\begin{align*}
  K&=\frac{\phi^{-1}_{rr}|\phi^{-1}_r|^{n-1}}{r^{n-1}(1+|\phi^{-1}_r|^2)^{\frac{n+2}{2}}}\\
  &\leq \frac{\phi^{-1}_{rr}}{((1-\sigma)R_0)^{n-1}(1+|\phi^{-1}_r|^2)^{\frac{3}{2}}}\\
  &=-\frac{\phi_{hh}}{(1-\sigma)^{n-1}R_0^{n-1}(1+\phi^2_h)^{\frac{3}{2}}}\\
  &\leq2R_0^{-n}(1-\sigma)^{-n}\delta,
\end{align*}
and
\begin{align*}
  H&=\frac{1}{\sqrt{1+(\phi^{-1}_r)^2}}\left((n-1)\frac{\phi^{-1}_r}{r}+\frac{\phi^{-1}_{rr}}{1+(\phi^{-1}_r)^2}\right)\\
  &=\frac{1}{\sqrt{1+\phi_h^2}}\left((n-1)\frac{1}{r}-\frac{\phi_{hh}}{1+\phi_h^2}\right)\\
  &\leq\frac{n-1}{(1-\sigma)R_0}+2\delta\leq \left(\frac{n-1}{1-\sigma}+2\right)R_0^{-1}\\
  &\leq\frac{2n}{(1-\sigma)R_0},
\end{align*}
where $\delta\leq \sigma R_0\leq 1/R_0$ is used in above inequality. 
Therefore, the following inequality holds
\begin{align*}
  F&=K^{s/n}G^{1-s}\leq K^{s/n}H^{1-s}
 \leq\left(\frac{2^{1/n}\delta^{1/n}}{(1-\sigma)R_0}\right)^s\left(\frac{2n}{(1-\sigma)R_0}\right)^{1-s}
  =\frac{2^{1-s+s/n}n^{1-s}}{(1-\sigma)R_0}\delta^{s/n}.
\end{align*}
In addition, the gradient function $v$ of $\phi^{-1}$ on $L_{[l-1, l)}(\Psi^{\delta, l}_t)=B_{\phi^{\delta, l}(l-1, t)}(0)\backslash B_{\phi^{\delta, l}(l, t)}(0)$ can be estimated as follows
\begin{align*}
  v=\sqrt{1+(\phi^{-1}_r)^2}=\sqrt{1+\frac{1}{4\delta^2(h-l)^2}}=\frac{\sqrt{1+4\delta^2(h-l)^2}}{2\delta(l-h)}\leq \frac{2}{\delta(l-h)},
\end{align*}
where $\delta\leq \sigma R_0\leq 1$ is used in the last inequality.
On the other hand, by
$$\phi^{-1}(\phi(h, t), t)=h,$$
we have
\begin{align*}
  \partial_t(\phi^{-1})=-\phi^{-1}_r\partial_t \phi=\frac{2^{(1-s+s/n)\beta+2}(1-\sigma)^{-\beta}n^{(1-s)\beta}R_0^{-\beta}\delta^{\frac{s\beta}{n}}}{2\delta(l-h)}\geq F^\beta v.
\end{align*}
Thus $\Psi^{\delta, l}_t$ is a supersolution of \eqref{1-1}.

Once the supersolution $\Psi^{\delta, l}_t$ is obtained, we can conclude with the arguments as in Theorem 5.4 in \cite{CDKL} that $B_{R_0}(x_0)\subset \Omega_{t_0}$ for any  $t_0\in(0, T)$.
\if fase
Assume there exists contact time $t^{\delta, l}\in(0, t_0]$ such that $\Sigma_t\cap{\rm Conv}(\Psi^{\delta, l}_t)=\varnothing$  for $t\in(0,  t^{\delta, l})$ and $\Sigma_{t^{\delta, l}}\cap{\rm Conv}(\Psi^{\delta, l}_{t^{\delta, l}})\neq\varnothing$, and  let $\Upsilon_{\delta, l}$ denote the contact set, that is,
\begin{align*}
  \Upsilon_{\delta, l}=\Sigma_{t^{\delta, l}}\cap{\rm Conv}(\Psi^{\delta, l}_{t^{\delta, l}})=\Sigma_{t^{\delta, l}}\cap\Psi^{\delta, l}_{t^{\delta, l}}.
\end{align*}
Since $\Sigma_t$ is a graph, it cannot contact $\Psi^{\delta, l}_{t^{\delta, l}}$ on $L_l(\Psi^{\delta, l}_{t^{\delta, l}})=B_{\phi^{\delta, l}(l, t^{\delta, l})}(0)$.
Then from the fact that
$$\Upsilon_{\delta, l}\subset\partial(\Psi^{\delta, l}_{t^{\delta, l}})=L_{l-1}(\Psi^{\delta, l}_{t^{\delta, l}})\cup L_l(\Psi^{\delta, l}_{t^{\delta, l}}),$$
we have $\Upsilon_{\delta, l}\subset L_{l-1}(\Psi^{\delta, l}_{t^{\delta, l}})$.

Suppose $(y_0, t^{\delta, l})\in \Upsilon_{\delta,l}$ is a contact point, then $|y_0|=\phi^{\delta, l}(l-1, t^{\delta, l})<R_0$ and the slope of the graph of $w$ at this point is at most equal to the slope of $\Psi^{\delta, l}_{t^{\delta, l}}$, hence $|Dw|_{(y_0, t^{\delta, l})}\leq \frac{1}{2\delta}$.
Let us define a set
$$A=\left\{(x, t): t\in[0, t_0], x\in\Omega_t, |x|\leq R_0, |D w|\leq \frac{1}{2\delta}\right\}.$$
The compactness of set $A$ implies the function  $w$ attains its maximum in $A$. Set $l_\delta=\max\{w(x, t), (x, t)\in A\}$.
Since $(y_0, t^{\delta, l})\in A$, we must have $l-1=w(y_0, t^{\delta, l})\leq l_\delta$.  Therefore, for any  $l>l_\delta+1$, the following fact holds
\begin{align*}
  \Upsilon_{\delta, l}\cap L_{l-1}(\Psi^{\delta, l}_{t^{\delta, l}})=\varnothing.
\end{align*}
Hence, we have $\Psi^{\delta, l}_{t_0}\prec \Sigma_{t_0}$ and $B_{\phi^{\delta, l}(l, t_0)}(0)=L_l(\Psi^{\delta, l}_{t_0})\prec L_l(\Sigma_{t_0})\preceq\Omega_{t_0}$.
And the result follows by letting $\delta$ to zero.
\fi
\end{proof}

{\renewcommand{\proofname}{\textbf{Proof of Theorem \ref{thm4}}}
\begin{proof}
It follows from Remark \ref{rem1} that $F=K^{s/n}G^{1-s}$ satisfies the Condition \ref{con1} for any $s\in(0, 1]$. Then by Theorem \ref{thm1}, there exists complete noncompact smooth strictly convex solution $\Sigma_t$ to \eqref{1-1}, which remains the graph for $t\in(0, T)$.
By Theorem \ref{thm3}, we have $\Sigma_t$ remains as a graph over the same domain $\Omega_0$ for all $t\in(0, T)$. As a conclusion,  $T=\infty$ and the assertion follows.
\end{proof}}


\end{document}